\documentclass[12pt,a4paper,american]{article}
\usepackage[latin9]{inputenc}
\usepackage{amsmath}
\usepackage{amssymb}
\usepackage{todonotes}

\makeatletter

\newcommand{\lyxaddress}[1]{
	\par {\raggedright #1
	\vspace{1.4em}
	\noindent\par}
}

\@ifundefined{date}{}{\date{}}

\usepackage{amsfonts}
\usepackage{amsthm}
\usepackage{mathrsfs}

\usepackage{relsize}

\newtheorem{theorem}{Theorem}
\newtheorem{proposition}[theorem]{Proposition}
\newtheorem{lemma}[theorem]{Lemma}
\newtheorem{corollary}[theorem]{Corollary}
\newtheorem*{lemma*}{Lemma}

\theoremstyle{remark}
\newtheorem{remark}[theorem]{Remark}
\newtheorem*{remark*}{Remark}
\newtheorem*{remarks*}{Remarks}
\newtheorem*{example*}{Example}
\newtheorem*{question*}{QUESTION}
\newtheorem*{conjecture*}{CONJECTURE}

\theoremstyle{definition}
\newtheorem{definition}[theorem]{Definition}
\newtheorem*{definition*}{Definition}
\newtheorem{notation}[theorem]{Notation}
\newtheorem*{notation*}{Notation}

\newcommand{\diag}{\mathop\mathrm{diag}\nolimits}

\makeatother

\makeatletter

\newcommand\appsection{\@startsection {section}{1}{\z@}%
{-3.5ex \@plus -1ex \@minus -.2ex}%
{2.3ex \@plus.2ex}%
{\normalfont\Large\bfseries\noindent%
Appendix.\hspace{0.5em}}}

\newtheorem{athm}{Theorem}[section]
\newtheorem{alem}[athm]{Lemma}

\newtheorem{acor}[athm]{Corollary}

\theoremstyle{remark}
\newtheorem{arem}[athm]{Remark}

\makeatother

\begin{document}
\title{Matrices with cyclically monotone rows and Cantor numeration systems}
\author{Pavel \v{S}\v{t}ov\'{\i}\v{c}ek and Edita Pelantov\'{a}  }
\maketitle

\lyxaddress{
\begin{center}
\noindent\emph{Department of Mathematics, Faculty of Nuclear Science}\\
\emph{Czech Technical University in Prague}\\
\emph{Trojanova 13, 120~00 Praha, Czech Republic}\\
\texttt{edita.pelantova@fjfi.cvut.cz, stovicek@fjfi.cvut.cz}
\end{center}
}

\begin{abstract}
We study a class of square matrices with non-negative elements which have
cyclically monotone rows in the sense that each row of a matrix from the class
consists of a cyclically non-increasing sequence of numbers starting
from a maximal element on the diagonal.
We prove that if every diagonal element is strictly larger than all other elements
in the respective row, then the matrix is regular.
This property enables us to solve an open problem that comes from the theory
of non-standard numeration systems, also called Cantor numeration systems.
The problem concerns a one-to-one relationship between Cantor real bases,
which are supposed to be alternate, that is, periodic with a period $p$,
and lists of $p$ sequences of non-negative integers satisfying the so-called
Parry condition.

\end{abstract}

\medskip


\medskip


\medskip

\section{Introduction}

Our main motivation in this paper is an open problem concerning the Cantor numeration systems.
To solve it, we introduce and study a special class of square matrices whose rows are
cyclically monotone. In more detail, each row of a square matrix from this class
consists of a cyclically non--increasing sequence of non-negative numbers
starting at the maximal element lying on the diagonal,
see Section \ref{sec:cyclicMonotone} for a formal definition.
Theorem \ref{thm:hlavni}, also in Section \ref{sec:cyclicMonotone}, contains the main result
of our article stating that the determinant of every matrix from this class is non-negative and,
under a somewhat stricter assumption, the matrix is even regular.
Let us note that the theorem may also be regarded as a partial generalization of a known result
valid for circulant matrices \cite[Prop. 24]{KrSi2012}.

Let us now briefly describe the nature and history of the above mentioned open problem.
The history goes back to Cantor, who was the first to use positional number systems
in which instead of one fixed integer $\beta >1$ chosen as a  base,
as customary in the decimal or binary system,
a sequence  $\mathcal{B}=(\beta_i)_{i\geq 0}$ of integers $\beta_i>1$ is employed.
In this numeration system, the $\mathcal{B}$-expansion of a number $x\in[0,1)$
- usually denoted  $(x)_{\mathcal{B}}$ - is a string
$a_0a_1 a_2 \cdots$ of integers $a_n$, $0\le a_n<\beta_n$, satisfying
 \begin{equation}\label{eq:Cantor-series}
    x = \sum_{n=0}^{+\infty}  \frac{a_n}{\beta_0\beta_1\cdots \beta_{n}}\,.
 \end{equation}
Cantor focused on the problem of specifying conditions that guaranty the number defined
by the series (\ref{eq:Cantor-series}) to be irrational.

It seems that Cantor's original ideas had to wait almost a hundred years
for further intensive development.
In 1954, Cantor's criterion for the irrationality of numbers defined by series
(\ref{eq:Cantor-series}) was generalized by Oppenheim \cite{Op1954}.
Further generalizations have been later presented, for example, in papers by Han\v{c}l and Tijdeman,
see \cite{HaTi2004,HaTi2008}. Another problem related to the Cantor series (\ref{eq:Cantor-series})
concerns statistical properties of the distribution of coefficients $a_n$. This topic
was already addressed in the 1950s by Erd\H{o}s and R\'enyi \cite{ErRe1959}, by Tur\'an \cite{Turan1956},
and later in the 1960s by \v{S}al\'at \cite{Salat1968}.
Furthermore, the question of the normality of numbers in this numeral system has become the focus
of recently published articles due to Airey and Mance and their collaborators, see \cite{AiMa2016}.
The numeration system, as originally introduced by Cantor, is a frequent subject of research
up to the present time, for instance, from the point of view of the theory of number approximations.
In this regard, let us just mention a very recent paper \cite{CHZL2026}.


While Cantor in \cite{Ca1869} admitted only integer sequences $\mathcal{B}$,
Caalim and Demegillo in \cite{CaDe2020} and, independently of them, Charlier and Cisternino
in \cite{ChaCi2021},
being inspired by the beta expansion of real numbers introduced by R\'enyi \cite{Renyi1957},
began to study {\it real Cantor numeration systems}.
In such systems $\mathcal{B}$ is a sequence of real numbers greater than one satisfying
$\prod_{n=0}^{+\infty}{\beta_n}=+\infty$.
Furthermore, condition $0\le a_n<\beta_n$ is replaced by another one that requires
$a_0a_1 a_2 \cdots$ to be the lexicographically greatest string
such that (\ref{eq:Cantor-series}) holds (in the case of integer bases, both conditions
are equivalent). $\mathcal{B}$ is called a \emph{Cantor real base}.
If the sequence $\mathcal{B}$ is purely periodic, the number system is called
{\it alternate}. As shown in \cite{CaDe2020} and \cite{ChaCi2021}, to an alternate base $\mathcal{B}$ with period $p$ one can assign $p$ integer sequences
${\bf d}^{(\ell)} = (d_{n}^{(\ell)})_{n\geq 0}$, $\ell =0,1,\ldots, p-1$,
with the help of which it can be decided whether a string $a_1a_2a_3\cdots$ represents
some number $x\in [0,1)$.

An important characteristic of sequences ${\bf d}^{(\ell)}$ assigned to a base $\mathcal{B}$
is that they satisfy the so-called Parry condition.
However, conversely, the question remained unanswered until recently
whether a base $\mathcal{B}$ exists for any list of sequences that satisfies
the Parry condition and, if so, whether it is uniquely given.
In a recently published article \cite{ChaKrMaPe2026}, the question of existence has been answered in the affirmative.
Furthermore, uniqueness has also been proven, but in some particular cases only, namely,
in \cite{ChCiMaPe2025}, in the case
where every ${\bf d}^{(\ell)}$ is eventually periodic and, in \cite{ChaKrMaPe2026},
in the case where every sequence  ${\bf d}^{(\ell)}={d}^{(\ell)}_0{d}^{(\ell)}_1\cdots$
starts with a number  ${d}^{(\ell)}_0\geq 2$.

In the present paper, we also contribute to the solution of this problem.
In Sections \ref{sec:alternatePowerSeries} and \ref{sec:sequencesParry},
with the aid of algebraic tools developed in Section \ref{sec:cyclicMonotone},
we are able to prove the uniqueness of the base  $\mathcal{B}$ for every list
${\bf d}^{(\ell)} = (d_{n}^{(\ell)})_{n\geq 0}$, $\ell =0,1,\ldots, p-1$,
of sequences that satisfy the Parry condition.

\section{Matrices with cyclically monotone rows}\label{sec:cyclicMonotone}
\begin{notation}
As common, $\mathbb{Z}_{p}$, $p\in\mathbb{N}$, stands for the cyclic
group of order $p$. As a set,
\[
\mathbb{Z}_{p}:=\{0,1,\ldots,p-1\},
\]
and the group operation in $\mathbb{Z}_{p}$ is the modular addition.

To simplify notation, {\em whenever an addition $j+k$, with $j,k\in\mathbb{Z}_{p}$,
occurs in any expression, the addition is understood as $j+k\mod p$}.

The symbol $\mathbb{N}$ stands for the set of positive integers, no special symbol
will be used for non-negative integers.
\end{notation}

In this section, we introduce and deal with a special class of square matrices.
\begin{definition}
A square matrix
\[
A=(a_{j,k})_{0\leq j,k<p},\ p\in\mathbb{N},
\]
is said to have \emph{cyclically monotone rows} if its entries are all non-negative
and its rows are cyclically non--increasing in the following sense:
\begin{equation}
 \forall j\in\mathbb{Z}_{p},\forall k\in\mathbb{Z}_{p}\setminus\{p-1\},\ a_{j,j+k}
 \geq a_{j,j+k+1}.\label{eq:c-monotone}
\end{equation}
If,  moreover, either
\begin{equation}\label{eq:strict-c-monotone-1}
p=1 \ \text{ and }\  a_{0,0} >0
\end{equation}
or
\begin{equation}\label{eq:strict-c-monotone}
p\geq 2 \ \text{ and }
 \ a_{j,j}>a_{j,j+1}\ \ \text{for every\ } j \in \mathbb{Z}_p
 \end{equation}
we say that $A$ is a matrix with \emph{strictly cyclically monotone rows}.
\end{definition}
\begin{example*}
Consider the matrices
\begin{equation*}
\left(\begin{matrix}5&4&3\\
1&2&2\\
2&2&3\\
\end{matrix}\right)\
\ \text{and}\ \
\left(\begin{matrix}5&4&4\\
1&2&1\\
2&2&3\\
\end{matrix}\right)\!.
\end{equation*}
The former has cyclically monotone rows, but not strictly, while
the latter one has strictly cyclically monotone rows.

\end{example*}
\begin{remark}\label{rem:simpleProperties}
Let us list some simple properties of matrices with (strictly) cyclically monotone rows
that we will use in some proofs to follow.
\begin{enumerate}
\item
Both conditions \eqref{eq:c-monotone} and \eqref{eq:strict-c-monotone}
are invariant with respect to
simultaneous cyclic permutation of rows and columns. Formally,  let
\begin{equation}
A'=(a'_{j,k})_{0\leq j,k<p},\ a'_{j,k}:=a_{j+1,k+1}\ \,\text{for all}\ j,k\in\mathbb{Z}_{p}.
\label{eq:cyclic-permute}
\end{equation}
If $A$ has cyclically monotone rows or strictly cyclically monotone rows,
then the same is true for $A'$. Moreover, $\det A'=\det A$.
\item
If $A_1, \ldots, A_m$ are matrices with cyclically monotone rows and
$\lambda_1, \ldots, \lambda_m$ are non-negative numbers, then the sum
\begin{equation*}
\sum_{k=1}^m \lambda_k A_k
\end{equation*}
is a matrix with cyclically monotone rows. If, moreover, there exists a summation index
$k$ such that $A_k$ has strictly cyclically monotone rows and $\lambda_k>0$
then the sum is also a matrix
with strictly cyclically monotone rows.
\item
Let $B$ be a principal submatrix of a $p\times p$ matrix $A$ (that is, a submatrix determined by a proper subset $I\subset\{0,1,\ldots,p-1\}$ and obtained from $A$
by removing rows and columns whose row index, respectively, column index belongs to $I$).
If $A$ has (strictly) cyclically monotone rows, then $B$ has (strictly) cyclically
monotone rows as well.
\item
Let $A$ be a matrix of size $p\in\mathbb{N}$ with (strictly) cyclically monotone rows and
\begin{equation*}
    c=(c_0,c_1,\ldots,c_{p-1}).
\end{equation*}
be a non--increasing sequence of non-negative numbers. Fix an index $j$, $0\leq j\leq p-1$, and modify $A$ by adding to its row with index $j$ sequence $c$. Denote the resulting
matrix by $B$. Then $B$ is a matrix with (strictly) cyclically monotone rows if and only if
 $j=0$ or $j\geq 1$ and
\begin{equation}\label{eq:b-mono}
    c_0 - c_{p-1} \leq a_{j,p-1} - a_{j,0}.
\end{equation}
Indeed, we only have to check the row of $B$ with index $j$. Then $b_{j,k}=a_{j,k}+c_k$.
The situation is obvious for $j=0$. Suppose $j\geq1$. Then surely
\begin{equation*}
    b_{j,j} \geq b_{j,j+1} \geq \ldots \geq b_{j,p-1} \ \text{ and }\
    b_{j,0} \geq b_{j,1} \geq \ldots \geq b_{j,j-1}.
\end{equation*}
Hence, it is sufficient and necessary to guarantee that $b_{j,p-1}\geq b_{j,0}$.
But this is exactly the condition (\ref{eq:b-mono}). Of course, if $a_{j,j}>a_{j,j+1}$
then $b_{j,j}>b_{j,j+1}$.
\end{enumerate}
\end{remark}

\begin{theorem}\label{thm:det}
If a square matrix $A$ of size $p\in\mathbb{N}$ has cyclically monotone rows,
particularly, if condition \eqref{eq:c-monotone} holds, then $\det A\geq 0$.

If, moreover,  $A$ satisfies (\ref{eq:strict-c-monotone-1}) or (\ref{eq:strict-c-monotone}),
that is, if $A$ has strictly cyclically monotone rows, then $\det A >0$.
\end{theorem}

\begin{proof}
We proceed by induction on the size of $A$. The case $p=1$ is trivial.


\medskip

Let $p\geq 2$. For an arbitrary but fixed  $\ell\in \{0,1,\ldots,p-2\}$, we consider an array
$A^{(\ell)}(x)=\big(a_{j,k}^{(\ell)}(x)\big)_{0\leq j,k<p}$ of matrix-valued functions of
a real variable $x$ defined by
\begin{equation}\label{eq:ajk-x}
    a_{j,k}^{(\ell)}(x) :=
    \begin{cases}
    x & \text{if}\ \,j=0,\,0\leq k\leq\ell,\\
    a_{j,k} & \text{otherwise}.
    \end{cases}
\end{equation}
Using the induction hypothesis we first show two auxiliary statements.

\medskip

\noindent {\bf Claim 1.} {\it
   The  function $\det A^{(\ell)}(x)$ of real variable $x$ is non--decreasing. If, moreover,
   $A$ satisfies \eqref{eq:strict-c-monotone} then the function is strictly increasing.
}

\medskip

To demonstrate the claim,  we realize that
$\det A^{(\ell)}(x)$ is a linear function of $x$, say  $\det A^{(\ell)}(x) = \alpha x+\beta$.
Let us show that coefficient $\alpha$ standing at $x$ is non-negative, respectively positive if \eqref{eq:strict-c-monotone} takes place.

Modify matrix \textbf{$A^{(\ell)}(x)$}
as follows. Subtract the first column from the columns with indices $k=1,\ldots,\ell$,
and add to the same columns the last column. Of course, this modification
does not influence the determinant. Let us denote the resulting matrix as
\begin{equation*}
 \tilde{A}^{(\ell)}(x)=\big(\tilde{a}_{j,k}^{(\ell)}(x)\big)_{0\leq j,k<p}
\end{equation*}
(in particular, for $\ell=0$ there is no modification at all and so
$\tilde{A}^{(0)}(x)=A^{(0)}(x)$). Then $\tilde{a}_{0,0}^{(\ell)}(x)=x$
and $x$ does not occur in any other entry of $\tilde{A}^{(\ell)}(x)$.

Hence the sought coefficient $\alpha$ equals the  minor of
matrix $\tilde{A}^{(\ell)}(x)$ corresponding to indices $(0,0)$, that is
$\alpha$ equals the determinant of the principal submatrix $\tilde{M}$ of
$\tilde{A}^{(\ell)}(x)$
obtained by removing the first row and column. Let us denote by $M$ the submatrix of $A$ also corresponding to indices $(0,0)$. We enumerate entries of both $M$ and $\tilde{M}$ by indices from
the range ${0,1,\dots,p-2}$. As observed in Remark \ref{rem:simpleProperties} ad(3),
$M$ has (strictly) cyclically monotone rows.

By construction of $\tilde{A}^{(\ell)}(x)$, for a fixed row index $j$,
$0\leq j\leq p-2$,
\begin{equation*}
    \tilde{m}_{j,k} = m_{j,k} + c_{k},\ \, k=0,1,\ldots,p-2,
\end{equation*}
where $m_{j,k}=a_{j+1,k+1}$ and
\begin{equation*}
    c_{k} =
    \begin{cases}
    a_{j+1, p-1}-a_{j+1, 0} & \text{if}\ \,0\leq k\leq \ell-1,\\
\noalign{\smallskip}
    0 & \text{ if }  \ell \leq k \leq p-2.
    \end{cases}
\end{equation*}
In particular, if $\ell=0$ then $\tilde{M}=M$ is a matrix with (strictly) cyclically
monotone rows.
Suppose $\ell\geq1$. We observe that $(c_0, c_1, \ldots, c_{p-2})$ is a non--increasing
sequence of non-negative numbers and, if $j\geq1$,
\begin{equation*}
    c_0 - c_{p-2} = a_{j+1, p-1}-a_{j+1, 0} \leq
    a_{j+1,p-1} - a_{j+1,1} =  m_{j,p-2} - m_{j,0}.
\end{equation*}
Thus we can apply Remark \ref{rem:simpleProperties} ad(4), repeatedly for each row index $j$,
to conclude that $\tilde{M}$ has (strictly) cyclically monotone rows.

We have found that, in any case, $\tilde{M}$ is a matrix with (strictly) cyclically
monotone rows. Therefore, by the induction hypothesis, $\alpha\geq0$ (respectively, $\alpha>0$). Claim 1 is proved.

\medskip

\noindent {\bf Claim 2. }  {\it
   Let $B=(b_{j,k})_{0\leq j,k<p}$ be the matrix defined by
   \begin{equation}\label{eq:def-B}
      b_{j,k} :=
      \begin{cases}
        a_{0,p-1} & \text{if}\ \,j=0,\\
        a_{j,k} & \text{otherwise}.
      \end{cases}
\end{equation}
Then $B$ has cyclically monotone rows and  $\det A\geq \det B$.
If $A$ has strictly cyclically monotone rows then $\det A >\det B$.
}

\medskip

To prove this claim we exploit Claim 1 saying that $\det A^{(\ell)}(x)$ is non--decreasing function
in the real variable $x$ for every $\ell$, $0\leq\ell\leq p-2$. Hence,
\[
\det A^{(\ell)}(a_{0,\ell}) \geq \det A^{(\ell)}(a_{0,\ell+1}).
\]
Moreover, it is clear from definition (\ref{eq:ajk-x}) that, for each $\ell$, $0\leq\ell\leq p-3$,
\[
A^{(\ell)}(a_{0,\ell+1})=A^{(\ell+1)}(a_{0,\ell+1}).
\]
Thus, for $0\leq\ell\leq p-3$,
\[
\det A^{(\ell)}(a_{0,\ell})\geq\det A^{(\ell+1)}(a_{0,\ell+1}).
\]
Furthermore, it is also clear that $A=A^{(0)}(a_{0,0})$ and $A^{(p-2)}(a_{0,p-1})=B$.
It follows that
\[
\det A =\det A^{(0)}(a_{0,0})\geq\det A^{(p-2)}(a_{0,p-2})\geq\det A^{(p-2)}(a_{0,p-1})=\det B.
\]

If $A$ satisfies even condition (\ref{eq:strict-c-monotone}), in particular if $a_{0,0} > a_{0,1}$, then
\[
\det A^{(0)}(a_{0,0})>\det A^{(0)}(a_{0,1}),
\]
and so $\det A >\det B$. The proof of Claim 2 is now complete.

\medskip

Under the  induction hypothesis we showed  that the determinant of any matrix $A$
with cyclically monotone rows is not smaller than the determinant of the matrix
obtained from A by replacing all entries in the first row by the smallest element in
the row. Obviously, the resulting matrix $B$ has again cyclically monotone rows.

\medskip

To complete the proof of the theorem, denote by $B'$ the matrix obtained from  $B$ by
the simultaneous cyclic permutation of rows and columns  described  in \eqref{eq:cyclic-permute}.
As stated in Item 1 of Remark \ref{rem:simpleProperties}, $B'$ has cyclically monotone
rows and $\det B = \det B'$.
Denote by  $B''$ the matrix obtained from $B'$ by replacing all entries in the first row
by the smallest element in the row. Then $B''$ has two constant rows, the first one and
the last one, and thus $\det B'' = 0$.

Claim $2$ applied to $A$ and $B'$ gives
\begin{equation*}
\det A \geq \det B = \det B ' \geq \det B'' = 0,
\end{equation*}
where the first inequality is strict if $A$ has strictly cyclically monotone rows.
\end{proof}

Recall that a matrix of the form
\begin{equation}\label{eq:circular}
C =
\left(\begin{array}{ccccc} c_0&c_1&\cdots & c_{n-2}&c_{n-1}\\
    c_{n-1}&c_{0}&\cdots &c_{n-3}&c_{n-2}\\
    \vdots &&&&\vdots\\
   c_{1}&c_2&\cdots&c_{n-1}&c_0
\end{array}\right)
\end{equation}
is called {\em circulant}. Let us also mention that, among other well known facts concerning
circulant matrices, an explicit formula exists for its determinant
(see, for example, Chapter 5 in \cite{Zhang2011}),
\begin{equation*}
\det C = \prod_{k=0}^{n-1} (c_0+c_1\omega_k+c_2\omega^{\,2}_k+\cdots +c_{n-1}\omega^{\,n-1}_k),
\quad \text{where } \omega_k = \exp(2\pi ik/n).
\end{equation*}

Regarding circulant matrices, Theorem \ref{thm:det} has a direct corollary if we assume
that the first row of a circulant matrix represents a monotone sequence. More
precisely, sticking to the notation of equation (\ref{eq:circular}), suppose that
\begin{equation}\label{eq:c-mono}
    c_0 \geq c_1 \geq c_2\geq \cdots \geq c_{n-1}\geq 0.
\end{equation}

\begin{corollary}
Let the entries of a circulant matrix $C$ defined in (\ref{eq:circular}) satisfy (\ref{eq:c-mono}).
Then  $\det C\geq0$. If, in addition, $c_0>c_1$ then $C$ is regular.
\end{corollary}

More is known in this particular case.

\begin{theorem}[\cite{KrSi2012}]\label{thm:regCircMono}
    Let the entries of a circulant matrix $C$ defined in (\ref{eq:circular}) satisfy (\ref{eq:c-mono}).
    Then $C$ is singular if and only if $n=kd$ where $k,d\in\mathbb{N}$, $d\geq2$, and
    the vector $(c_0,c_1,\ldots,c_{n-1})$ consists of $k$ consecutive constant blocks of length $d$.
\end{theorem}

We just note that the easy corollary following from Theorem \ref{thm:det} is in agreement
with Theorem \ref{thm:regCircMono}.

\section{Alternate power series}\label{sec:alternatePowerSeries}

\begin{definition}
Let $p$ be a positive integer. A mapping  $f:[0,1)^p \to \mathbb{R}$ is said to be
{\it alternate power series of  type $p$} if
for every $\vec{y}=(y_0,y_1, \ldots, y_{p-1})^T\in [0,1)^p$
\begin{equation}\label{eq:def-alternate-power-series}
f(\vec{y})= f(y_0, y_1, \ldots, y_{p-1})
= \sum_{n=0}^{+\infty} a_n\prod_{i=0}^{n}y_{i\!\!\!\mod p}
\end{equation}
where $(a_n)_{n\geq 0}$ is a  bounded sequence  of non-negative numbers such that $a_0>0$ and
$a_n > 0$ for infinitely many indices $n$.

\end{definition}
Obviously, $f$ is differentiable in any point $\vec{y}\in [0,1)^p$.

\begin{lemma}\label{lem:TvarFunkce}
Let $f:[0,1)^p \to \mathbb{R}$ be  an alternate power series of  type $p \in \mathbb{N}$.
Then for all $\vec{y}=(y_0,y_1, \ldots, y_{p-1})^T\in (0,1)^p$ and
$j\in \{0,1,\ldots, p-1\}$ the following inequalities take place
\begin{equation*}
y_j\,\frac{\partial f(\vec{y})}{\partial y_j}  \geq 0
\qquad \text{and }\qquad
y_j\,\frac{\partial f(\vec{y})}{\partial y_j}
- y_{j+1}\,\frac{\partial f(\vec{y})}{\partial y_{j+1}}
\ \ \left\{\begin{array}{ll}
    > 0  & \text{ if } j=0;\\
    &\\
    \geq 0  & \text{ if } j>0.
\end{array}\right.
\end{equation*}
\end{lemma}

\begin{proof}
For a non-negative integer $n$ let us write $n=kp+r$, with $k$ and $r$ being non-negative integers, $r<p$.
Then
\begin{equation*}
    \prod_{i=0}^{n}y_{i\!\!\!\!\mod p}
    = \prod_{m=0}^{p-1}y_m^{\,\, \sigma_m},\ \text{where}\ \sigma_m
    = \begin{cases}
       k+1 & \text{if}\ \,0\leq j\leq r,\\
       k & \text{if}\ \,r< j\leq p-1.
       \end{cases}.
\end{equation*}
For the purposes of this proof it suffices to notice that, for any given $n$,
the respective sequence of powers $\{\sigma_0,\sigma_1,\ldots,\sigma_{p-1}\}$ is
non-increasing and, in addition, $\sigma_0>\sigma_1$ if $n=0$.
%
%
Since
\begin{equation*}
y_j\,\frac{\partial }{\partial y_j}\,\prod_{m = 0}^{p-1} y_m^{\,\,\sigma_m}
= \sigma_j\,\prod_{m = 0}^{p-1}y_m^{\,\,\sigma_m},
\end{equation*}
all $y_m$ are positive, all $a_n$ are non-negative, and $a_0 >0$, the lemma follows.
\end{proof}

\begin{definition}\label{def:PSI}
Let $f^{(0)}, f^{(1)},  \ldots, f^{(p-1)} $ be a $p$-tuple of alternate power series
of type $p\in\mathbb{N}$.
For $i \in \{0, 1, \ldots, p-1\}$, we define  $\psi_i: [0,1)^p\to\mathbb{R}$  as follows:
\begin{equation}\label{eq:def-psi}
\psi_i(\vec{y}) = f^{(i)}(y_i,y_{i+1}, \ldots, y_{i+p-1})
\text{ \quad for every  $\vec{y} = (y_0, y_1, \ldots, y_{p-1}) \in [0, 1)^p$ }.
\end{equation}
The mapping   $\Psi=(\psi_0, \psi_1, \ldots, \psi_{p-1})^T: [0,1)^p \to \mathbb{R}^p$
is called {\it the vector-valued function associated with  $f^{(0)}, f^{(1)},  \ldots, f^{(p-1)} $}.
\end{definition}

\begin{lemma}\label{lem:DerivativeCyclic}
Let $\Psi$  be as in Definition \ref{def:PSI}.
Then the Jacobian matrix $J_{\Psi}(\vec{y})$ of $\Psi$ multiplied from the right
by the diagonal matrix
$\diag(y_0, y_1, \ldots, y_{p-1})$ has strictly cyclically monotone rows
at  every $\vec{y} \in (0,1)^p$.
\end{lemma}

\begin{proof}
Since  $\psi_i(\vec{y}) = f^{(i)}(y_i, y_{i+1}, \ldots, y_{i+p-1})$, Lemma \ref{lem:TvarFunkce} implies
\begin{equation}
y_i\,\frac{\partial \psi_i(\vec{y})}{\partial y_i}> y_{i+1}\,
\frac{\partial \psi_i(\vec{y})}{\partial y_{i+1}}\geq \cdots \geq
y_{i+p-1}\,\frac{\partial \psi_i(\vec{y})}{\partial y_{i+p-1}} \geq 0.
\end{equation}
This means exactly that the product of matrices $J_{\Psi}(\vec{y})$ and
$\diag(y_0, y_1, \ldots, y_{p-1})$ is a matrix with strictly cyclically monotone rows.
\end{proof}

\begin{corollary}\label{col:positiveMinor}
Let $\Psi$ be as in Definition \ref{def:PSI}, and $\vec{y} \in (0,1)^p$.
Then every principal minor of the Jacobian matrix $J_{\Psi}(\vec{y})$ is positive.
\end{corollary}

\begin{proof}
For ease of notation, in this proof we put $J:=J_{\Psi}(\vec{y})$,
$D:=\diag(y_0, y_1, \ldots, y_{p-1})$ and $C:=JD$. Let $I\subset\{0,1,\ldots,p-1\}$
be an arbitrary proper subset and $J'$ be the principal submatrix of $J$
corresponding to $I$, as described in Remark \ref{rem:simpleProperties} ad(3).
Similarly, $D'$ is the principal submatrix of $D$ corresponding to $I$, and $C'$
is the principal submatrix of $C$ corresponding to $I$. Because $D$
is diagonal, we clearly have $C':=J'D'$.

Lemma \ref{lem:DerivativeCyclic} guarantees that $C$ has strictly cyclically monotone
rows and, as observed in Remark \ref{rem:simpleProperties} ad(3), the same is true
for every principal submatrix of $C$.
Therefore, by Theorem \ref{thm:det}, $\det(C')>0$. As all components of $\vec{y}$ are
positive, $\det(D')>0$. Consequently, $\det(J')>0$.
\end{proof}

The paper \cite{GaNi1965} by  Gale and Nikaido is devoted to  mappings whose Jacobian
matrix has the property described in Corollary \ref{col:positiveMinor}.
Let us quote their main result.

\begin{theorem}[\cite{GaNi1965}]\label{thm:GaleNikaido}
Let $\Phi: \Omega\to\mathbb{R}^p$ be a differentiable mapping on  a closed rectangular
region of $\mathbb{R}^p$. If every leading principal minor of the Jacobian matrix
$J_{\Phi}(\vec{y})$  is positive at every $\vec{y} \in \Omega$
then $\Phi$ is injective on $\Omega$.
\end{theorem}

The following theorem is a straightforward consequence  of the previous theorem and Corollary \ref{col:positiveMinor}.

\begin{theorem}\label{thm:hlavni}
The vector-valued function $\Psi$ associated with alternate power series of type~$p$
is injective on $(0,1)^p$.
\end{theorem}

\begin{proof}
Choose  $\varepsilon \in (0, \frac12) $  and denote
$\Omega_{\varepsilon} := [\varepsilon, 1-\varepsilon]^p$. Corollary \ref{col:positiveMinor}
and Theorem \ref{thm:GaleNikaido} imply injectivity of $\Psi$ on $\Omega$. As $\varepsilon$
can be arbitrarily small, the theorem is proved.
\end{proof}

Additionally, we can propose an alternative proof of Theorem \ref{thm:hlavni}
that is based on a mean value theorem for vector-valued functions due to McLeod \cite{McLeod1964}
rather than on Theorem \ref{thm:GaleNikaido}.

\begin{theorem}[\cite{McLeod1964}]\label{thm:McLeod}
Let $f:[0, 1] \to\mathbb{R}^p$ be a continuous vector-valued function, and assume that
$f$ is continuously differentiable on $(0, 1)$. Then there exist $p$ points
$\xi_1, \xi_2, \ldots, \xi_p \in (0, 1)$ and $p$ non-negative numbers
$\lambda_1, \lambda_2, \ldots, \lambda_p$ such that
\begin{equation*}
\lambda_1 + \lambda_2 + \cdots + \lambda_p = 1
\end{equation*}
and
\begin{equation*}
f(1) - f(0) = \sum_{k=1}^p \lambda_k f'(\xi_k).
\end{equation*}
\end{theorem}

\begin{remark}\label{rem:McLeold-extend}
    Theorem \ref{thm:McLeod} has an immediate extension that can be formulated as follows.

{\em Let $\Omega$ be an open convex region in $\mathbb{R}^m$ and $F:\Omega\to\mathbb{R}^p$
be a continuously differentiable vector-valued function. Then for every couple of points
$\vec{a},\vec{b}\in\Omega$, $\vec{a}\neq\vec{b}$, there exist $p$ points $\vec{\eta}_1,\vec{\eta}_2,\ldots,\vec{\eta}_p$
lying in the interior of the segment $[\,\vec{a}\,,\vec{b}\,]$ and $p$ non-negative numbers
$\lambda_1, \lambda_2, \ldots, \lambda_p$ such that
\begin{equation*}
\lambda_1 + \lambda_2 + \cdots + \lambda_p = 1
\end{equation*}
and
\begin{equation}\label{eq:Fb-minus-Fa}
F(\vec{b}) - F(\vec{a}) = \bigg(\sum_{k=1}^p \lambda_k\, J_F(\vec{\eta}_k)\bigg)(\vec{b}-\vec{a}),
\end{equation}
where $J_F(\vec{\eta})$ is the Jacobian matrix of $F$ at a point $\vec{\eta}$.
}

Of course, to see it, it suffices to consider the vector-valued function 
$f:[0,1]\to\mathbb{R}^p$, $f(t):=F\big((1-t)\vec{a}+t\vec{b}\,\big)$.
\end{remark}

Combining the observation from Remark \ref{rem:McLeold-extend} with Theorem \ref{thm:det} we have
the following theorem.

\begin{theorem}\label{thm:cyclic-mono-injective}
Let $\Omega$ be an open convex region in $\mathbb{R}^p$ and $F:\Omega\to\mathbb{R}^p$
be continuously differentiable.
If the Jacobian matrix $J_F(\vec{a})$ has strictly cyclically monotone rows at every point $\vec{a}\in\Omega$ then $F$ is injective.
\end{theorem}

\begin{proof}
Suppose that $\vec{a},\vec{b}\in\Omega$, $\vec{b}-\vec{a}\neq\vec{0}$. Equation (\ref{eq:Fb-minus-Fa})
is applicable to $F$ and, as noted in Remark \ref{rem:simpleProperties} ad(2), the convex combination
of Jacobian matrices on the RHS is also equal to a matrix with strictly cyclically monotone rows.
Hence, by Theorem \ref{thm:det}, this convex combination represents a regular $p\times p$ matrix.
Therefore, $F(\vec{b})-F(\vec{a})\neq\vec{0}$.
\end{proof}

\begin{proof}[An alternative proof of Theorem \ref{thm:hlavni}]
Consider the mapping $T:(-\infty,0)^p\to (0,1)^p$,
\begin{equation*}
T(x_0,x_1,\ldots,x_{p-1}):=\big(\exp(x_0),\exp(x_1),\ldots,\exp(x_{p-1})\big)^T.
\end{equation*}
$T$ is one-to-one and thus $\Psi$ is injective on $(0,1)^p$ if and only if $\Psi\circ T$ is injective
on $(-\infty,0)^p$. Regarding the respective Jacobi matrices we have
\begin{equation*}
    J_{\Psi\circ T}(\vec{x}) = J_\Psi(\vec{y})\,\diag(y_0,y_1\ldots,y_{p-1})\ \,
    \text{where}\ \,\vec{y} = (y_0,y_1\ldots,y_{p-1}) = T(\vec{x}).
\end{equation*}
By Lemma \ref{lem:TvarFunkce}, $J_{\Psi\circ T}(\vec{x})$ is
a matrix with strictly cyclically monotone rows for every $\vec{x}\in(-\infty,0)^p$.
Theorem \ref{thm:cyclic-mono-injective} then tells us that $\Psi\circ T$ is injective, and so is $\Psi$.
\end{proof}

\section{Sequences satisfying the Parry condition}\label{sec:sequencesParry}

Let $\mathcal{B} = (\beta_n)_{n\geq 0}$ be a Cantor real base, as explained in Introduction.
The definition of $\mathcal{B}$-expansion preserves the ordering of numbers, i.e. if $0\leq x<y <1$, then
$(x)_{\mathcal{B}}\prec_{lex} (y)_{\mathcal{B}} $. This property enables us to define
the {\it quasi-greedy $\mathcal{B}$-expansion of unity }
\begin{equation*}
d_{\mathcal B}^*(1) := \lim_{x\to 1_-}(x)_{\mathcal{B}}
\end{equation*}
where the limit is meant in the product topology.
Let us recall some properties of $d_{\mathcal B}^*(1) = d_0d_1d_2\ldots$.
\begin{enumerate}
    \item  $d_n$ is a non-negative integer for every index $n$;
    \item $d_0>0$ and $d_n >0$ for infinitely many indices $n$;
    \item  $\displaystyle 1 =\, \sum_{n=0}^{+\infty}  \frac{d_n}{\beta_0\beta_1\cdots \beta_{n}}$.
\end{enumerate}

To formulate the Parry condition, we will use the notation
\begin{equation*}
\sigma(\mathcal{B}) := (\beta_{n+1})_{n\geq 0}
\ \ \text{where}\ \ \mathcal{B} = (\beta_n)_{n\geq 0},
\end{equation*}
that is, $\sigma(\mathcal{B})$ is a shift of a sequence $\mathcal{B}$.

\begin{theorem}[\cite{CaDe2020}, \cite{ChaCi2021}]\label{thm:Parry-pre}
Let $\mathcal{B} = (\beta_n)_{n\geq 0}$ be a Cantor real base. A string $a_0a_1a_2\cdots$
of non-negative integers is the $\mathcal{B}$-expansion of a number $x \in [0,1)$ if and only if
\begin{equation}\label{eq:Parry-pre}
a_{n}a_{n+1}a_{n+2}\cdots \prec_{lex}  d_{\sigma^{n}(\mathcal B)}^*(1) \quad
\text{for every } \ n\geq 0.
\end{equation}
\end{theorem}
Obviously, if the base $\mathcal B$ is alternate of period $p$, then
$d_{\sigma^{n}(\mathcal B)}^*(1) = d_{\sigma^{n+p}(\mathcal B)}^*(1)$ for every $n\geq 0$.

Recall that the numeration system introduced by A. R\'enyi in \cite{Renyi1957} is a special case of the
alternate numeration system for $p=1$. This means that Theorem \ref{thm:Parry-pre}, if specialized to
$p=1$, is concerned with R\'enyi numeration systems. In this particular case, it has been derived
by Parry in \cite{Parry1960}. But Parry's name is now associated with condition (\ref{eq:Parry-pre})
in the general case as well.

\begin{example*} The classical decimal system  is an alternate base system  of period $p=1$ and $d_{\mathcal B}^*(1) = d_{\sigma^{n}(\mathcal B)}^*(1) = 999999\cdots$ for every non-negative integer $n$.

For the original Cantor base $\mathcal{B} = (\beta_n)_{n\geq 0}$ consisting of integers
$\beta_n>1$, the quasi-greedy expansions of unity have the form
$ d_{\sigma^{n}(\mathcal B)}^*(1) = (\beta_{n}-1)(\beta_{n+1}-1)(\beta_{n+2}-1)\cdots$
for every $n\geq 0$.
\end{example*}

\begin{definition}\label{def:Parry}
Let $p\in \mathbb{N}$ and ${\bf d}^{(\ell)} = d^{(\ell)}_0d^{(\ell)}_1d^{(\ell)}_2\cdots$
be a sequence of non-negative integers for each $\ell = 0,1,\ldots, p-1$.
For every index $n\geq 0$, define  ${\bf d}^{(n)} := {\bf d}^{(\ell)}$ where $n = \ell \mod p$.

We say that the list ${\bf d}^{(0)}, {\bf d}^{(1)}, \ldots, {\bf d}^{(p-1)}$ satisfies
{\em the Parry condition} if for each $\ell=0,1,\ldots, p-1$ the following inequalities
take place:
\begin{itemize}
  \item   $d^{(\ell)}_0>0$,
  \item $d^{(\ell)}_n>0$ for infinitely many indices $n$,
   \item
  $d_{n}^{(\ell)}d_{n+1}^{(\ell)}d_{n+2}^{(\ell)} \cdots  \preceq_{lex} {\bf d}^{(\ell+n)}$
  for every $n\geq 0$.
\end{itemize}
\end{definition}
\begin{remark}
We emphasize that the lexicographic inequality in the Parry condition enforces the boundedness
of the sequences ${\bf d}^{(\ell)}$ for every $\ell = 0,1,\ldots, p-1$. In particular,
each term $d^{(\ell)}_n$  is less than or equal to
$H = \max\{d^{(0)}_0,d^{(1)}_0,\ldots, d^{(p-1)}_0\}$.
\end{remark}

It is well known that the Parry condition holds for every list of sequences that is
associated with an alternate Cantor real base.

\begin{proposition}[\cite{CaDe2020}, \cite{ChaCi2021}]
Let $\mathcal{B}$ be an alternate real base of period $p \in \mathbb{N}$. Then the list
$d_{\mathcal B}^*(1), d_{\sigma(\mathcal B)}^*(1), \ldots, d_{\sigma^{p-1}(\mathcal B)}^*(1) $
satisfies the Parry condition.
\end{proposition}

The inverse relationship is a more demanding question. The existence problem has already been solved.

\begin{theorem}[\cite{ChaKrMaPe2026}]\label{thm:Existence}
Let ${\bf d}^{(0)}, {\bf d}^{(1)}, \ldots, {\bf d}^{(p-1)}$  satisfy the Parry condition.
Then there exists an alternate real base $\mathcal{B}$ of period $p$ such that
$d_{\sigma^{\ell}(\mathcal B)}^*(1)= {\bf d}^{(\ell)}  $ for each $\ell = 0,1,\ldots, p-1$.
\end{theorem}

Let us mention that in \cite{ChaKrMaPe2026}, some combinatorial properties of the set of so-called
$\mathcal{B}$-integers were used to prove the existence  of a base  $\mathcal{B}$ associated with
a given list of sequences satisfying the Parry condition.
We are able to provide an alternative proof of the existence theorem. Since our proof relies solely on some
topological arguments we postpone it to Appendix.

In the article \cite{ChaKrMaPe2026}, the question of the ambiguity of the base $\mathcal{B}$
has been addressed, but only for lists
${\bf d}^{(0)}, {\bf d}^{(1)}, \ldots, {\bf d}^{(p-1)}$ of a  particular form. Having
equipped ourselves with the tools developed in the preceding sections, we are now ready to deal
with the general case.

\begin{theorem}
Let $\mathcal{B}_1$  and $\mathcal{B}_2$ be two alternate real bases of period $p$ with the same
list of quasi-greedy expansions of unity. Then  $\mathcal{B}_1=\mathcal{B}_2$.
\end{theorem}

\begin{proof}
Let ${\bf d}^{(\ell)} = (d^{(\ell)}_n)_{n\geq 0}$, for $\ell=0,1,\ldots, p-1$ form the list
of quasi-greedy expansions of unity. In particular,  $d^{(\ell)}_0 >0$
for every $\ell = 0,1,\ldots, p-1$.
Let $\Psi$ be the vector-valued function associated with the alternate power series
$f^{(0)}, f^{(1)}, \ldots, f^{(p-1)}$,  where
\begin{equation*}
f^{(\ell)}(\vec{y})= f^{(\ell)}(y_0, y_1, \ldots, y_{p-1}) = \sum_{n=0}^{+\infty} {d^{(\ell)}_n}\prod_{i=0}^{n}y_{i\!\!\!\!\mod p}
\end{equation*}
Theorem \ref{thm:hlavni} implies that there exists at most one $p$-tuple
$\vec{y} = (y_0,y_1,\ldots, y_{p-1}) \in (0,1)^p$ such that
for each $\ell = 0,1,\ldots, p-1$,
\begin{equation*}
1=\psi_{\ell}(\vec{y}) = f^{(\ell)}(y_{\ell}, y_{\ell+1}, \ldots, y_{\ell+p-1}) =
\sum_{n=0}^{+\infty} d^{(\ell)}_n\prod_{i=0}^n{y_{\ell+i \!\!\mod p}}.
\end{equation*}
Defining $(\beta_0, \beta_1, \ldots, \beta_{p-1}):=(y_0^{\,-1}, y_1^{\,-1}, \ldots, y_{p-1}^{\,-1})$,
we deduce that there  exists at most one $p$-tuple $(\beta_0, \beta_1, \ldots, \beta_{p-1})$
of numbers $>1$ such that the alternate base $(\beta_{n})_{n\geq 0}$ of period $p$ satisfies
\begin{equation*}
1=\sum_{n=0}^{+\infty}  \frac{d^{(\ell)}_n}{\beta_{\ell}\beta_{\ell+1}\cdots \beta_{\ell+n}}\quad
\text{ for each $\ell =0,1,\ldots, p-1.$}
\end{equation*}
This proves the theorem.
\end{proof}

\section{Concluding comments}


To show the irrationality of a real number, one must prove that the number does not belong
to the subfield $\mathbb{Q}$. This concept can be generalized to the irrationality with respect
to other subfields of real numbers, e.g., to the extension of the field $\mathbb{Q}$ by
an algebraic number $\alpha$, which is usually denoted as $\mathbb{Q}(\alpha)$.
It would be interesting to find criteria for this type of irrationality using a real Cantor basis.

However, it should be noted that when applying the irrationality criterion, a Cantor base
$\mathcal{B}$ is used for this purpose that is not periodic, i.e., $\mathcal{B}$ is not alternate.
In such a case, the list of quasi-greedy expansions of unity associated with $\mathcal{B}$
is infinite and also satisfies  Parry's condition. But the question we have solved here
for alternate Cantor bases has not been addressed yet at all in the case of non-periodic
bases.

\newpage

\makeatletter
\renewcommand{\@seccntformat}[1]{}
\makeatother

\setcounter{section}{0}
\renewcommand{\thesection}{\Alph{section}}

\setcounter{equation}{0}
\renewcommand{\theequation}{\Alph{section}\arabic{equation}}

\appsection{The range of $\Psi$}

Let us discuss the range of the mapping $\Psi=(\psi_{0},\psi_{1},\ldots,\psi_{p-1})^{T}:[0,1)^{p}\to\mathbb{R}^{p}$
associated with a $p$-tuple of alternate power series of type $p\in\mathbb{N}$,
as introduced in Definition \ref{def:PSI}.
For each $i=0,1,\ldots, p-1$ we assume that the coefficients $a^{(i)}_n$ in
$$\psi_i(\vec{y}) = \sum_{n=0}^{+\infty} a^{(i)}_n\prod_{k=0}^{n}y_{i+k \!\!\!\mod p}$$
have these two properties:
\begin{description}
    \item[$\mathcal{P}1$: ] \ \ $(a^{(i)}_n)_{n \geq 0}$  is a  bounded sequence of non-negative numbers,
    \item [$\mathcal{P}2$: ] \ \ $a^{(i)}_0 \geq 1 $ and  $a^{(i)}_0 + a^{(i)}_n > 1 $ for some  $n\geq 0$.
\end{description}

Our goal is to prove the following theorem.
\begin{athm}\label{thm:range}
$\Psi\big((0,1)^{p}\big)\supset(0,1]^{p}$.
\end{athm}

Before we start with the proof, we derive an auxiliary statement and recall
two theorems from topology on which our proof is based.

\begin{alem}\label{lem:krychle}
Let $\delta \in (0,1/2)$ and $r \in (0,1)$. Then there exist $\varepsilon', \varepsilon'' \in (0,\delta) $
such that the $p$-dimensional cube $[r,1]^p$ and the image $\Psi(\Sigma)$ are disjoint
where $\Sigma$
is the boundary of the cube $[\varepsilon', 1-\varepsilon'' ]^p$.
\end{alem}

\begin{proof}
By $\mathcal{P}1$, $(a^{(i)}_n)_{n \geq 0}$ is bounded, which implies that each of
the series $\psi_{i}$ is convergent even if $y_{j}\in[0,1)$ for only one index $j$,
$0\leq j\le p-1$, while all other variables $y_{k}$, $k\neq j$, are equal to $1$.
Therefore, $\Psi$ is continuous in the domain
$$
[0,1]^{p}\setminus\{(1,1,\ldots,1)\}
$$
Moreover, we can factor $y_{i}$ out of $\psi_{i}$, which means that there exists
another power series $\omega_{i}$ such that
 $\psi_{i}(\vec{y})=y_{i}\,\omega_{i}(\vec{y})$.
Consequently,
\begin{equation*}
\psi_i(\underbrace{1,\ldots,1}_{i\, times},\  0,
\underbrace{1,\ldots, 1 }_{p-1-i\,times})= 0.
\end{equation*}
This allows us to choose $\varepsilon'\in (0,\delta)$ such that
\begin{equation*}
\psi_i(\underbrace{1,\ldots,1}_{i\, times},\  \varepsilon',
\underbrace{1,\ldots, 1 }_{p-1-i\,times}) < r.
\quad \text{for each } i,\,1\le i\le p-1.
\end{equation*}
By $\mathcal{P}2$, the function $\psi^{(i)}(y_0,y_1, \ldots, y_{p-1})$ is non-decreasing separately in each variable.
Hence,
\begin{equation}\label{eq:Psi-under-cube}
\psi_{i}(\vec{y}) <r\ \ \text{for every\ } \ \vec{y}\in(0,1]^{i}\times \{\varepsilon'\} \times(0,1]^{p-1-i}\,.
\end{equation}

Properties $\mathcal{P}1$ and $\mathcal{P}2$ guarantee that
$\omega_{i}(\varepsilon',\varepsilon',\ldots,\varepsilon') > \omega_{i}(\vec{0}) \geq 1$.
Thus, we can  choose $\varepsilon''\in(0,\delta)$ sufficiently small so that
\begin{equation*}
(1-\varepsilon'')\,\omega_{i}(\varepsilon',\varepsilon',\ldots,\varepsilon') > 1
\quad \text{for each } i,\,1\le i\le p-1.
\end{equation*}
Hence, for $i\in \{0,1,\ldots, p-1\}$  we can estimate
\begin{equation}\label{eq:Psi-out-cube}
\psi_{i}(\vec{y})\geq (1-\varepsilon'')\,\omega_{i}(\varepsilon',\varepsilon',\ldots,\varepsilon')>1
\ \ \text{for every}\ \ \vec{y}\in[\varepsilon',1]^{i}\times\{1-\varepsilon''\}\times[\varepsilon',1]^{p-1-i}.
\quad
\end{equation}

Now it suffices to realize  that a vector $\vec{y} = (y_0,y_1, \ldots, y_{p-1})^T$
belongs to $\Sigma$ if and only if $y_i\in [\varepsilon', 1-\varepsilon'']$
for each $i\in \{0,1,\ldots, p-1\}$ and $y_j$ equals $\varepsilon'$ or $1-\varepsilon''$
for at least one index $j$.
Inequalities  \eqref{eq:Psi-under-cube} and \eqref{eq:Psi-out-cube} mean that for every
point $\vec{y}\in \Sigma$. $\psi_j(\vec{y}) <r$ or $\psi_j(\vec{y}) >1$ hold.
Consequently, $\Psi(\vec{y})$ does not belong to the cube $[r,1]^p$.
\end{proof}

We shall need a classical theorem due to E.~L.~J.~Bouwer that is recalled below
\cite{Brouwer-domain}.

\begin{athm}[Brouwer's Theorem on the Invariance of Domain]\label{thm:invariance}
Let $U$ be an open subset of the Euclidean space\textbf{ $\mathbb{R}^{p}$},
$p\in\mathbb{N}$, and $\Phi\!:U\to\mathbb{R}^{n}$ be an injective
continuous mapping. Then $\Phi(U)$$\subset\mathbb{R}^{n}$ is an
open subset and $\Phi$ is even a homeomorphism of $U$ onto $\Phi(U)$.
\end{athm}

%

We now have all the ingredients needed to prove the inclusion
$\Psi\big((0,1)^{p}\big)\supset(0,1]^{p}$.

\medskip

\begin{proof}[Proof of Theorem \ref{thm:range}]
Due to properties $\mathcal{P}1$ and $\mathcal{P}2$, whenever $\vec{y} \in (0,1)^p$,
all components of the vector $\Psi(\vec{y})$ are positive.
As $\Psi$ is continuous on $[0,1)^p$ and $\Psi(\vec{0}) = \vec{0}$, there exists
$\vec{z} \in (0,1)^p$ such that $\Psi(\vec{z}) \in (0,1)^p$, too.

Choose $\delta\in (0, 1/2)$  and $r\in (0,1)$ satisfying
$\vec{z} \in (\delta,1-\delta)^p$ and $\Psi(\vec{z}) \in (r,1)^p$.
Let us consider $\varepsilon'$ and $\varepsilon''$ assigned by Lemma \ref{lem:krychle}
to such  $\delta$ and $r$.

Let $K:=[\varepsilon', 1-\varepsilon'' ]^p$.
We again denote by $\Sigma$ the boundary of the cube $K$ and by ${\rm Int}(K)$
its interior $(\varepsilon', 1-\varepsilon'' )^p$.
By the choice of $\varepsilon'$, $\varepsilon''$ and $r$, $\vec{z}\in{\rm Int}(K)$
and $\Psi(\vec{z})\in[r,1]^p$. Hence
\begin{equation}\label{eq:intersect}
[r,1]^p \cap\, \Psi\big({\rm Int}(K)\big) \neq \emptyset.
\end{equation}

Recall that alternate power series  $\Psi$ is continuously differentiable on $(0,1)^p$
and, by Theorem \ref{thm:hlavni}, also injective on this domain.
%
On the other hand, by Lemma \ref{lem:krychle}, $[r,1]^p$ is disjoint with $\Psi(\Sigma)$.
Thus we find that
\begin{equation*}
\mathbb{R}^p = \big(\mathbb{R}^p \setminus \Psi(K)\big) \cup \Psi(\Sigma)
\cup \Psi\big({\rm Int}(K)\big).
\end{equation*}
where the union is disjoint, and
\begin{equation*}
[r,1]^p \subset \mathbb{R}^p \setminus \Psi(\Sigma)
= \big(\mathbb{R}^p \setminus \Psi(K)\big) \cup \Psi\big({\rm Int}(K)\big).
\end{equation*}
By Brouwer's Theorem on the Invariance of Domain,
$\Psi\big({\rm Int}(K)\big)$ is an open subset of $\mathbb{R}^p$.
Note also that $\Psi(K)$ is compact.
This means that the connected set $[r,1]^p$ is contained in a disjoint union
of two open sets and, therefore, it must be entirely contained within one of them.
In regard of (\ref{eq:intersect}), we conclude that
\begin{equation}\label{eq:r}
[r,1]^p \subset  \Psi\big({\rm Int}(K)\big) \subset \Psi\big((0,1)^p\big).
\end{equation}

Let us emphasize that when choosing  $ r$ in $ (0,1)$ we had to respect only one condition,
namely that $\Psi(\vec{z}) \in (r,1)^p$. Of course, $r$ can be replaced by any smaller
positive number. In other words, \eqref{eq:r} holds true for every $r\in (0,1)$, i.e.,
$(0,1]^p \subset \Psi \big((0,1)^p\big)$, as desired.
\end{proof}

We have an obvious corollary.
\begin{acor} There exists $\vec{y}\in(0,1)^{p}$ such that
\[
\psi_{0}(\vec{y})=\psi_{1}(\vec{y})=\cdots=\psi_{p-1}(\vec{y})=1.
\]
\end{acor}

\begin{arem}\label{rem:jinydukaz}
Note that Properties $\mathcal{P}1$ and $\mathcal{P}2$ are fulfilled for the alternate power
series associated with a list ${\bf d}^{(0)}, {\bf d}^{(1)}, \ldots, {\bf d}^{(p-1)}$ of
sequences  satisfying the Parry condition, see Definition \ref{def:Parry}.
If we use the solution $\vec{y} = (y_0,y_1, \ldots, y_{p-1})^T$ from the previous corollary and put
$\beta_i :=  y_i^{\,-1}$, $0\le i\le p-1$, we get a new proof of Theorem \ref{thm:Existence}.
\end{arem}

\end{document}